\documentclass[<opyions>]{elsarticle}
\usepackage{amssymb}
\usepackage{graphicx}
\usepackage{epstopdf}
\usepackage{subfigure}
\usepackage{extarrows}
\usepackage{fancyhdr}
\usepackage{amsthm}
\usepackage{extarrows}
\usepackage{caption}
\usepackage{float}
\usepackage{algorithm}
\usepackage{algorithmic}
\usepackage{multirow}
\usepackage{amsmath}
\usepackage[numbers]{natbib}
\usepackage{graphicx}
\usepackage{amsmath} 
\usepackage{amsfonts} 
\usepackage{mathrsfs} 
\usepackage{booktabs} 
\usepackage{color} 

\usepackage{color}
\usepackage{hyperref}

\usepackage[toc,page,title,titletoc,header]{appendix}
\usepackage{appendix}
\theoremstyle{definition}

\theoremstyle{definition}
\usepackage{hyperref}
\usepackage{latexsym, bm}
\usepackage{fancyhdr}
\usepackage{mathrsfs}
\usepackage{wasysym}
\usepackage{float}
\usepackage{titlesec}
\usepackage{pgfplots}
\usepackage{tikz}
\usepackage{subfigure}
\usepackage{natbib}
\biboptions{numbers,sort&compress}

\usetikzlibrary{arrows,shapes,positioning}
\usetikzlibrary{decorations.markings}
\tikzstyle arrowstyle=[scale=1]
\tikzstyle directed=[postaction={decorate,decoration={markings,
    mark=at position .65 with {\arrow[arrowstyle]{stealth}}}}]
\tikzstyle reverse directed=[postaction={decorate,decoration={markings,
    mark=at position .65 with {\arrowreversed[arrowstyle]{stealth};}}}]

\newtheorem{lemma}{Lemma}[section]
\newtheorem{theorem}{Theorem}[section]

\newtheorem{remark}{Remark}[section]

\newtheorem{example}{Example}[section]

\allowdisplaybreaks
\begin{document}
	
\begin{frontmatter}
\title{{\bf A new analytical technique of the fully implicit Crank-Nicolson discontinuous Galerkin method for the Ginzburg-Landau Model}}

\author{Xianxian Cao}

\author{Zhen Guan\corref{cor1}}
\ead{zhenguan1993@foxmail.com}

\author{Junjun Wang}

\cortext[cor1]{Corresponding author.}

\address{School of Mathematics and Statistics, Pingdingshan University, Pingdingshan, 467000, China}

\begin{abstract}
In this paper, a fully implicit Crank-Nicolson discontinuous Galerkin method is proposed for solving the Ginzburg-Landau equation. By leveraging a novel analytical technique, we rigorously establish the unique solvability of the constructed numerical scheme, as well as its unconditionally optimal error estimates under both the \(L^2\)-norm and the energy norm. The core of the proof hinges on the \(L^2\)-norm boundedness of the numerical solution and the refined estimation of the cubic nonlinear term. Finally, two numerical examples are presented to validate the theoretical findings.
\end{abstract}


\end{frontmatter}
\thispagestyle{empty}

\numberwithin{equation}{section}
\section{Introduction}\label{section01}

The Ginzburg-Landau equation is a class of important complex-valued nonlinear partial differential equations, first proposed by physicists Vitaly Ginzburg and Lev Landau in the 1950s to model low-temperature superconductivity. This model finds significant applications in fields including condensed matter physics, nonlinear optics, and nonequilibrium hydrodynamic systems \cite{Guo2020}. Under different application contexts, the Ginzburg-Landau equation exhibits distinct formulations, such as the nonlinear Ginzburg-Landau equation with a Caputo derivative \cite{Chen2023}, the Ginzburg-Landau-Schrödinger equation \cite{Yao2026}, the Kuramoto-Tsuzuki complex equation with strong nonlinear effects \cite{Zhang2025}, the dynamical Ginzburg-Landau equation under the Coulomb gauge \cite{Gao2023}, and the time-dependent Ginzburg-Landau equation under the temporal gauge \cite{Ma2023}.

In this paper, we are devoted to the discontinuous Galerkin approximation of the following two-dimensional Ginzburg-Landau equation with cubic nonlinearity in a bounded convex polygonal domain $\Omega \subset \mathbb{R}^2$:   
\begin{align}
u_t - (\nu + \mathrm{i}\alpha)\Delta u + (\kappa + \mathrm{i}\beta)\vert u \vert^2 u - \gamma u = 0, \quad (\boldsymbol{x},t) \in \Omega \times (0,T], \label{202510302231}
\end{align}
subject to the homogeneous Dirichlet boundary condition
\begin{align}
u(\boldsymbol{x},t) = 0,  \quad (\boldsymbol{x},t) \in \partial\Omega\times(0,T], \label{10302241}
\end{align}
and the initial condition
\begin{align}
u(\boldsymbol{x},0) = u^0(\boldsymbol{x}), \quad \boldsymbol{x} \in \bar{\Omega}, \label{10302242}
\end{align}
where $\nu>0,\kappa>0,\alpha,\beta,\gamma$ are given real constants, $\partial\Omega$ represents the boundary of $\Omega$, $\boldsymbol{x} = (x, y)$, $\Delta$ is the Laplace operator, $\mathrm{i}=\sqrt{-1}$ denotes the imaginary unit and $u^0(\boldsymbol{x})$ is a sufficiently smooth complex-valued function with a zero boundary trace. 

Due to the significance of this equation, numerous scholars have proposed various effective numerical methods for solving this model, including the finite difference method (FDM) \cite{ZhangYang2025,XuChang2011,HuChenChang2015}, the finite element method (FEM) \cite{LiCaoZhang2019,Shi2020,Yang2025}, the meshless method \cite{Li2025}, and the discontinuous polygonal method \cite{GuanCao2025}. More precisely, Zhang and Yang \cite{ZhangYang2025} proposed a linearized three-layer finite difference scheme for the nonlinear variable-coefficient Kuramoto-Tsuzuki complex equation. The scheme employs standard central finite differences for spatial discretization, combines central averaging with central differences for temporal discretization, and adopts a semi-implicit linearization approach to handle the nonlinear terms.
Li and Cao \cite{LiCaoZhang2019} developed a linearized Galerkin finite element method for solving the two-dimensional and three-dimensional Kuramoto-Tsuzuki equation. By classifying and refining the treatment of the nonlinear terms, the authors derived the optimal \(L^2\)-norm error estimate for the numerical solution. Li et al. \cite{Li2025} studied a fast element-free Galerkin (EFG) method for solving the nonlinear complex Ginzburg-Landau equation. Using the space-time error splitting technique, they achieved optimal order convergence. Recently, Guan and Cao \cite{GuanCao2025} proposed a discontinuous Galerkin method on polygonal meshes, implemented via the weighted \(\theta\) scheme.
They proved the boundedness and optimal convergence under various norms. From the aforementioned literature review, it can be observed that most numerical studies employ semi-implicit schemes, wherein the nonlinear terms are treated explicitly. While such schemes afford computational convenience, they compromise certain inherent properties of the original equation to a limited extent.
Motivated by this consideration, a fully implicit numerical scheme is proposed in this paper, which combines the discontinuous Galerkin method for spatial discretization with the Crank-Nicolson scheme for temporal discretization. To the best of our knowledge, such a fully implicit Crank-Nicolson discontinuous Galerkin scheme for the Ginzburg-Landau equation has not been reported in the existing literature.

As is well established in the field, two distinct classes of numerical schemes are commonly employed for nonlinear partial differential equations: fully implicit scheme and linearized scheme. For linearized scheme, the unconditionally optimal error estimates are usually derived using the space-time error splitting technique proposed by Li and Sun \cite{Li2012}. The core idea is to first introduce a time-discrete system, and then use the inverse inequality to prove the boundedness of the numerical solution in the infinity norm. More precisely,
\begin{align*}
	\|u_{h}^{n}\|_{\infty}&\leq \|R_{h}U^{n}\|_{\infty}+\|u_{h}^{n}-R_{h}U^{n}\|_{\infty}\\
	&\leq C\|U^{n}\|_{2}+Ch^{-\frac{d}{2}}\|u_{h}^{n}-R_{h}U^{n}\|\\
	&\leq C\|U^{n}\|_{2}+Ch^{2-\frac{d}{2}}\\
	&\leq C,
\end{align*}
where $U^n$ is the numerical solution of time-discrete system at the $n$-th level, $R_h$ is a suitable elliptic projection operator, $h$ and \(\tau\) denote the spatial and temporal mesh sizes, respectively. Another method employed in the literature is to prove the boundedness of numerical solution in \(H^1\)-norm by means of mathematical induction \cite{Sun2017}. 
However, those methods are only applicable to linearized numerical scheme, and fail for fully implicit algorithm.  For fully implicit scheme, there are generally two analytical methods, one of which is the duality argument technique proposed by Feng et al. \cite{FengLiMa2021}. The key points of its proof lie in the Schaefer's fixed point theorem and the discrete Ladyzhenskaya's inequality. This method has been successfully applied to solving other nonlinear partial differential equations, including the Schrödinger-Helmholtz equation \cite{GuoChenZhou2025} and the coupled Schrödinger equation \cite{LiLiSun2023}. Another approach involves the cutoff function technique \cite{WangWang2023}, the core idea lies in truncating the nonlinear term to satisfy the global Lipschitz condition. Both of the aforementioned techniques merely establish the existence of solution to the fully discrete numerical scheme and the uniqueness of  numerical solution in a neighborhood of exact solution. However, they fail to address the global unique solvability of the scheme.

To establish the global unique solvability of the proposed fully implicit numerical scheme, we develop a novel analytical approach for its proof, drawing on the core concepts of the two aforementioned techniques. This constitutes the primary contribution of this paper.
First, we establish the uniform boundedness of the numerical solution in \(L^2\)-norm. Subsequently, by utilizing this result and conducting a refined estimate of the cubic nonlinear term, we derive an estimation formula for the numerical solution and Ritz projection of the true solution. Finally, we analyze the time step \(\tau\) and spatial step $h$ under two cases (\(\tau \leq h\) and \(\tau > h\)), obtain the boundedness of the numerical solution in energy norm, and further establish the unique solvability and optimal convergence rate of the fully discrete numerical scheme. We conjecture that the analytical technique proposed in this paper can be extended to other types of nonlinear partial differential equations.

The structure of this paper is organized as follows. In Section \ref{section2}, we introduce the notations, lemmas, and the numerical scheme under investigation. Section \ref{section3} is dedicated to the theoretical analysis of the numerical scheme, encompassing the global existence and uniqueness of the numerical solution, as well as its convergence analysis. In Section \ref{section4}, two numerical examples are presented to validate the theoretical results. Finally, Section \ref{section5} provides a summary of the entire paper and an outlook for future research.

\section{The fully implicit Crank-Nicolson discontinuous Galerkin Scheme 
}\label{section2}
In this section, we first introduce key notations and preliminary lemmas employed throughout the paper, upon which we formulate the fully implicit Crank-Nicolson discontinuous Galerkin method.

Through the paper, we adopt the standard Sobolev space notation commonly used in the discontinuous finite element literature. For any integer $m\geq 1$, let $|\cdot|_{H^{m}(\Omega)}$ and $\|\cdot\|_{H^{m}(\Omega)}$ denotes the seminorm and norm of the Sobolev space $H^{m}(\Omega)$, respectively. The norm of the complex-valued Banach space $L^p(\Omega)$ is denoted by $\|\cdot\|_{p}$, where 
$1 \leq p \leq \infty$. When $p=2$, the space $L^2(\Omega)$ is a Hilbert space with respect to the following inner product and norm:
\begin{align}
	(u,v) = \int_{\Omega}uv^{*}\text{d}\boldsymbol{x},\quad \|u\| = \sqrt{(u,u)},\quad u,v\in L^2(\Omega),\notag
\end{align}
where $v^{*}$ denotes the conjugate of $v$. When no ambiguity arises, this paper adopts similar notations for the inner product and norm of vector functions. Finally, for a strongly measurable function $v:(0,T)\rightarrow X$, where $X$ is a complex Banach space, we introduce the Bochner space defined as 
$$
L^{p}({0,T;X})=\left\{v:\int_{0}^{T}\|v\|_{X}^p\text{d}t<\infty\right\},
$$
with the associated norm 

$$\|v\|_{L^{p}({0,T;X})} = \left(\int_{0}^{T}\|v\|_{X}^p\text{d}t\right)^{\frac{1}{p}},
$$
when $p=\infty$, the norm and integral defined above are implicitly subject to the usual modification.

Let \(\mathcal{K}_h\) denote a sequence of quasi-uniform triangulations of \(\Omega\), where $h$ denotes the mesh size (defined as the diameter of the largest element in \(\mathcal{K}_h\)). For simplicity, we assume \(\mathcal{K}_h\) is conforming: the intersection of any two distinct elements in the triangulation is either the empty set, a single vertex, or a complete edge. We denote the set of all edges in the partition $\mathcal{K}_{h}$ by $\mathcal{E}_{h}$. Moreover, interior edges are collected in the set $\mathcal{E}^{i}_{h}$ and boundary edges in $\mathcal{E}^{b}_{h}$. For every element $K\in \mathcal{K}_{h}$, $h_{K}$ is the diameter of $K$. Similarly, define $h_{E}$ as the length of each edge $E$. We orient each edge $E$ with a fixed normal vector $\boldsymbol{n}_E$. When $E$ is on the boundary of the region $\Omega$, take $\boldsymbol{n}_E$ as the unit outward normal vector of the boundary $\partial\Omega$. Let $ v $ be a scalar-valued function defined on $ \Omega $, and suppose $ v $ is sufficiently smooth. This smoothness ensures that $v $ has a possibly two-valued trace on every $ E \in \mathcal{E}^{i}_{h} $. Then, for all $E\in \mathcal{E}^{i}_{h}$ shared by two adjacent elements $K_1$ and $K_2$, the average and jump are defined as 
\begin{align}
\{v\} = \frac{1}{2}(v\vert_{K_1}+v\vert_{K_2}),\quad [v] = v\vert_{K_1}-v\vert_{K_2},\notag
\end{align}
respectively, where the outer normal vector of $K_1$ on the edge $E$ is consistent with $\boldsymbol{n}_E$. The concepts of average and jump can be extended to the boundary edge $E\in\mathcal{E}^{b}_{h}$:
$$
\{v\} = [v] = v\vert_{K}.
$$
Next, for a fixed positive integer $k$, we give the complex-valued discontinuous Galerkin approximation subspace $V_{h}^{k}$ defined as 
\begin{align}
	V_{h}^{k}=\{v_{h}\in L^{2}(\Omega): v_{h}|_{K}\in \mathbb{P}_{k}(K),~\forall~K\in\mathcal{K}_{h}\},\notag 
\end{align}
where $\mathbb{P}_{k}(K)$ denotes the space of polynomials of total degree less than or equal to $k$. In the finite-dimensional space $	V_{h}^{k}$, the discontinuous Galerkin norm is defined as follows 
\begin{align}
\|v_{h}\|_{\text{DG}}:= \bigg(\sum_{K \in \mathcal{K}_h} \int_K |\nabla v_h|^2\mathrm{d}\boldsymbol{x}+ \sum_{E \in \mathcal{E}_h} \frac{1}{h_E} \int_E  [v_h] ^2\mathrm{d}s\bigg)^{1/2},\label{2507281430}
\end{align}
where $\vert \cdot \vert$ denotes the Euclidean norm in $\mathbb{C}^2$.  
Based on the arguments in Reference \cite{PietroErn2010}, it can be concluded that the following Sobolev embedding inequality holds
\begin{align}
\| v_h\|_{p} \leq C  \| v_h \|_{\text{DG}}, \quad \forall ~v_h \in V_{h}^{k},\quad 1 \leq p < \infty. \label{11042314}
\end{align}
We further present the discrete Ladyzhenskaya's inequality \cite{GazcaOrozcoKaltenbach2025} which is essential to the derivation process of the optimal error estimate
\begin{align}
	\| v_h\|_{4} \leq C  \| v_h \|_{\text{DG}}^{1/2}\| v_h\|^{1/2}, \quad \forall ~v_h \in V_{h}^{k}. \label{202509042305}
\end{align}
We also recall the following global inverse inequality \cite{Chave2016,GuanCao2025}:
\begin{align}
&\|v_h\|_{\text{DG}}\leq Ch^{-1}\|v_h\|,\quad \forall~v_h\in V_{h}^{k}.	\label{10312320}
\end{align}
Moreover, let us define the discrete bilinear form corresponding to the continuous bilinear form $(\nabla u,\nabla v)$
\begin{align}
	a_h(v_h, w_h) &:= \int_\Omega\nabla_h v_h \cdot \nabla_h w_h^* \, \mathrm{d}\boldsymbol{x} 
	- \sum_{E \in \mathcal{E}_h} \int_E \{ \nabla_h v_h \} \cdot \boldsymbol{n}_E [w_h^*] \, \mathrm{d}s -\sum_{E \in \mathcal{E}_h} \int_E [v_h] \{ \nabla_h w_h^* \} \cdot \boldsymbol{n}_E \, \mathrm{d}s 
	\notag \\
	&\quad+ \sum_{E \in \mathcal{E}_h} \frac{\lambda}{h_E}\int_E [v_h][w_h^*] \, \mathrm{d}s, \quad \forall~v_h, w_h\in V_{h}^{k}, \notag
\end{align}
where $\nabla\cdot$ is the piecewise-defined gradient operator and  parameter $\lambda$, known as the penalty term, is a sufficiently large non-negative real number in this context.
It is a well-established fact that the discrete bilinear form introduced earlier satisfies the coercivity and continuity properties outlined below
\begin{align}
	C_1 \| v_h \|_{\text{DG}}^{2} \leq a_{h}(v_{h},v_{h})=:\|v_{h}\|^2_{a_h} \leq C_2 \| v_h \|_{\text{DG}}^{2}, \quad \forall~ v_h\in V_h^{k}.\label{2508251529}
\end{align}
On this basis, we define the elliptic projection operator $R_{h}:  H^{k+1}(\Omega)\rightarrow V_{h}^{k}$ satisfying
\begin{align}
	a_{h}(R_{h}u,v_{h})=-(\Delta u,v_h),\quad \forall v_{h}\in V_{h}^{k}.\label{2510311736}
\end{align}
From the argument in \cite{DiPietroErn2012,Riviere2008}, for every $u\in  H^{k+1}(\Omega)$, we know the following boundedness and approximation properties
\begin{align}
	&\|u-R_hu\|_{\infty}+\|R_hu\|_{\infty}\leq C\|u\|_{k+1},\label{202510311606}\\
	&\|R_{h}u-u\|+h\|R_{h}u-u\|_{\text{DG}}\leq Ch^{k+1}\|u\|_{k+1}. \label{2510281415}
\end{align}

Let $\{t_{n}:n=0,1,2,\cdots,N\} $ be the uniform partition of the interval $[0, T]$ with the time step $\tau=T/N$ and denote $t_{n-\frac{1}{2}}:=(n-\frac{1}{2})\tau$. For a series of functions $\{\phi^{n}:n=0,1,2,\cdots,N\}$ defined on the domain $\Omega$, we write 
\begin{align*}
D_\tau\phi^{n} &:= \frac{\phi^{n}-\phi^{n-1}}{\tau}, \quad \hat{\phi}^{n}:= \frac{\phi^{n}+\phi^{n-1}}{2},\quad n = 1,2,\cdots,N-1,N.
\end{align*}
Using some of the notations given in the preceding text, the fully implicit Crank-Nicolosn discontinuous Galerkin algorithm is to seek $u_{h}^{n}\in V_{h}^{k}$ such that for $n=1, 2, \cdots, N-1, N$ 
\begin{align}
	&(D_\tau u_{h}^{n},v_{h})+(\nu + \mathrm{i}\alpha)a_h(\hat{u}_h^n, v_h)+(\kappa + \mathrm{i}\beta)(\vert \hat{u}^n_h \vert^2 \hat{u}_h^n,v_h)-\gamma (\hat{u}_h^n,v_h)=0,\quad \forall ~v_h\in V_h^k,\label{07271656}
\end{align}
where $u_{h}^0$ is taken as the Ritz projection $R_hu^0$ of $u^0$. 

\begin{lemma}\label{lemma1}
(Brouwer Fixed Point Theorem \cite{Akrivis1991, Akrivis1993})
	Let $(H, (\cdot, \cdot))$ be a finite-dimensional inner product space, with $\|\cdot\|$ the associated norm, and $\Pi : H \to H$ be continuous. Assume moreover that there exists an $a > 0$, for any $z \in H$ and $\|z\| = a$, it holds that
	\[
	\operatorname{Re}(\Pi(z), z) \geqslant 0.
	\]
	Then there is a $z' \in H$ satisfying $\|z'\| \leqslant a$ such that $\Pi(z') = 0$.
\end{lemma}
\begin{lemma}\label{lemma2}
\text{(Discrete Gronwall inequality \cite{HeywoodRannacher1990})} Let $A\geq 0, B \geq 0$, $\{\eta^i\}_{i=1}^{N}$ and $\{\xi^i\}_{i=1}^{N}$ be two sequences of non-negative real numbers satisfying 
$$
\quad \eta^n + \tau\sum_{i=1}^n \xi^i\leq A + B\tau \sum_{i=1}^n \eta^i, \quad 1 \le n \le N.
$$
Then, the following holds when $\tau \le \frac{1}{2B}$ 
$$\eta_n +\tau\sum_{i=1}^n \xi^i\leq A \exp(2Bn\tau), \quad 1 \le n \le N.$$
\end{lemma}

\begin{lemma}\label{lemma3}
\text{(Discrete Gronwall inequality \cite{Sun2023})} Suppose $\{F^n\}_{n=0}^N$ is a nonnegative sequence, $C$ and $G$ are two non-negative constants satisfying
\begin{align*}
	F^{n} \leqslant (1 + C\tau)F^{n-1} + \tau G, \quad n = 1,2,\cdots,N,
\end{align*}
then
\begin{align}
	F^n \leqslant \mathrm{e}^{Cn\tau}\left(F^0 + \frac{G}{C}\right), \quad n = 1,2,\cdots,N.
\end{align}
\end{lemma}
\begin{lemma}\label{lemma4}	
(Transfer Inequality \cite{Wang2014}) For a normed linear space $\mathcal{V}$ with the norm $\|\cdot\|_{\mathcal{V}}$, and for elements $v^0, v^1, \ldots, v^N \in \mathcal{V}$, the result below holds
\begin{align}
	\|v^n\|_{\mathcal{V}} \leq 2 \sum_{m=1}^{n} \left\|\hat{v}^{m}\right\|_{\mathcal{V}} + \|v^0\|_{\mathcal{V}}, \quad 1 \leq n \leq N.\label{2507302329}	
\end{align}
\end{lemma}

\section{Numerical analysis of the fully implicit discontinuous Galerkin scheme} \label{section3}
In this section, we initiate the theoretical analysis of the proposed numerical scheme. To this end, we first establish the existence of the numerical solution, followed by a rigorous proof of its \(L^2\)-norm boundedness. Based on this result, we further derive the boundedness of the numerical solution under the energy norm, as well as the optimal order error estimates in both the \(L^2\)-norm and the energy norm.
\subsection{Existence of the discontinuous Galerkin solution}  
\begin{theorem}\label{theorem1}
When $\tau\gamma<2$, the fully discrete discontinuous Galerkin algorithm \eqref{07271656} has a solution.
\begin{proof}
Assume that the numerical solution $u_{h}^{n-1}$ at the $(n-1)$-th layer has been obtained. Then \eqref{07271656} can be interpreted as a nonlinear system in terms of the average solution $\hat{u}_{h}^n$:
\begin{align}
	&\frac{2}{\tau}\left(\hat{u}_{h}^n-u_{h}^{n-1},v_{h}\right)+(\nu + \mathrm{i}\alpha)a_h(\hat{u}_h^n, v_h)+(\kappa + \mathrm{i}\beta)(\vert \hat{u}^n_h \vert^2 \hat{u}_h^n,v_h)-\gamma (\hat{u}_h^n,v_h)=0,\quad \forall ~v_h\in V_h^k.\label{10311810}
\end{align}
Obviously, if we can prove the existence of a solution to equation \eqref{10311810}, then we can easily know that $u_h^n=2\hat{u}_h^n-u_h^{n-1}$ is 
a solution of \eqref{07271656}. For this purpose, define the operator $\Pi: V_h^k\rightarrow V_h^k $ by 
\begin{align}
\left(\Pi (w), v_h\right) = \frac{2}{\tau}\left(w-u_{h}^{n-1},v_{h}\right)+(\nu + \mathrm{i}\alpha)a_h(w, v_h)+(\kappa + \mathrm{i}\beta)(\vert w \vert^2 w,v_h)-\gamma (w,v_h), \quad \forall ~v_h\in V_h^k.\label{11021104}
\end{align}
Next, we will prove that this operator is continuous. In fact, given $w_1,w_2\in V_h^k$, it can be deduced from \eqref{11021104} that 
\begin{align}
\left(\Pi (w_1)-\Pi (w_2), v_h\right)&=\frac{2}{\tau}\left(w_1-w_2,v_h\right)+(\nu + \mathrm{i}\alpha)a_h(w_1-w_2, v_h)+(\kappa + \mathrm{i}\beta)(\vert w_1 \vert^2 w_1-\vert w_2 \vert^2 w_2,v_h)\notag\\
&\quad-\gamma (w_1-w_2,v_h),\quad\forall ~v_h\in V_h^k.\label{11021111}
\end{align}
By taking $v_h=\Pi (w_1)-\Pi (w_2)$ in \eqref{11021111} and employing the Cauchy-Schwartz inequality, we find that 
\begin{align}
\|\Pi (w_1)-\Pi (w_2)\|^2 &\leq C\|w_1-w_2\|_{a_h}\|\Pi (w_1)-\Pi (w_2)\|_{a_h}+C\left\|\vert w_1 \vert^2 w_1-\vert w_2 \vert^2 w_2\right\|\|\Pi (w_1)-\Pi (w_2)\|\notag\\
&\quad+\left(C+\frac{2}{\tau}\right)\|w_1-w_2\|\|\Pi (w_1)-\Pi (w_2)\|\notag\\
&\leq C\|w_1-w_2\|_{\text{DG}}\|\Pi (w_1)-\Pi (w_2)\|_{\text{DG}}+C\left\|\vert w_1 \vert^2 w_1-\vert w_2 \vert^2 w_2\right\|\|\Pi (w_1)-\Pi (w_2)\|\notag\\
&\quad+\left(C+\frac{2}{\tau}\right)\|w_1-w_2\|\|\Pi (w_1)-\Pi (w_2)\|\notag\\
&\leq \frac{C}{h^2}\|w_1-w_2\|\|\Pi (w_1)-\Pi (w_2)\|+\left(C+\frac{2}{\tau}\right)\|w_1-w_2\|\|\Pi (w_1)-\Pi (w_2)\|\notag\\
&\quad+C\left(\|w_1\|^2_{\infty}+\|w_1\|_{\infty}\|w_2\|_{\infty}+\|w_2\|^2_{\infty}\right)\|w_1-w_2\|\|\Pi (w_1)-\Pi (w_2)\|,\label{11021656}
\end{align}
where we have also used \eqref{2508251529}, \eqref{10312320} and the following elementary inequality 
\begin{align}
\left||z_1|^2z_1-|z_2|^2z_2\right|\leq \left(|z_1|^2+|z_1||z_2|+|z_2|^2\right)|z_1-z_2|,\quad \forall~z_1,z_2\in \mathbb{C}. \label{202511042246}
\end{align}
By dividing both sides of inequality \eqref{11021656} by $\|\Pi (w_1)-\Pi (w_2)\|$ simultaneously, it holds that 
\begin{align}
\|\Pi (w_1)-\Pi (w_2)\|\leq C\left(\frac{1}{ h^2}+1+\frac{2}{\tau}+\|w_1\|^2_{\infty}+\|w_1\|_{\infty}\|w_2\|_{\infty}+\|w_2\|^2_{\infty}\right)\|w_1-w_2\|.\label{11021700}
\end{align}
From \eqref{11021700}, we can easily see that the operator $\Pi$ is continuous. Furthermore, choosing $v_h=w$ in \eqref{11021104}, taking the real part of both sides and using the Cauchy-Schwartz inequality give that 
\begin{align}
\text{Re}\left(\Pi(w),w\right) &=  \frac{2}{\tau}\left(\|w\|^2-\text{Re}(u_{h}^{n-1},w)\right)+\nu\|w\|_{a_h}+\kappa\|w\|_{4}^4-\gamma\|w\|^2\notag\\
&\geq \frac{2}{\tau}\left( \|w\|^2-\|u_{h}^{n-1}\|\|w\|\right)-\gamma\|w\|^2\notag\\
&\geq \frac{2}{\tau}\left[\left(1-\frac{1}{2}\gamma\tau\right)\|w\|-\|u_{h}^{n-1}\|\right]\cdot\|w\|.\label{11021726}
\end{align}
Obviously, when $\tau\gamma< 2 $ and $\|w\|=\frac{\|u_h^{n-1}\|}{1-\frac{1}{2}\tau\gamma}+1$, it follows that 
\begin{align}
\text{Re}\left(\Pi(w),w\right)\geq 0.\notag
\end{align}
By the Brouwer fixed point theorem described in Lemma \ref{lemma1}, we can conclude that there exists a $w'$ satisfying $\|w'\|\leq \frac{\|u_h^{n-1}\|}{1-\frac{1}{2}\tau\gamma}+1$ such that 
$$
\Pi(w')=0.
$$
Therefore, there is a solution to \eqref{07271656}. 
All this completes the proof. 
\end{proof}
\end{theorem}
\subsection{Boundedness of the discontinuous Galerkin solution} 
\begin{theorem}\label{theorem2}
Suppose $u_h^n$ is the solution of the fully discrete numerical scheme \eqref{07271656}. Then when $\tau\gamma\leq 1$, we have 
\begin{align*}
\|u_h^{n}\|\leq e^{2T\max\{0,\gamma\}}\|u_h^0\|, \quad 1\leq n \leq N.
\end{align*}
\end{theorem}
\begin{proof}
Taking $v_h=\hat{u}_h^n$ in \eqref{07271656}, we have 
\begin{align}
	&\left(D_\tau u_{h}^{n},\hat{u}_h^n\right)+(\nu + \mathrm{i}\alpha)\|\hat{u}_h^n\|_{a_h}^2+(\kappa + \mathrm{i}\beta)\|\hat{u}_h^n\|_4^4-\gamma \|\hat{u}_h^n\|^2=0.\label{11021940}
\end{align}
Noticing the fact that 
\begin{align*}
\text{Re}\left(D_\tau u_{h}^{n},\hat{u}_h^n\right) = \frac{1}{2\tau}\left(\|u_h^n\|^2-\|u_h^{n-1}\|^2\right). 
\end{align*}
Taking the real part of both sides of equation \eqref{11021940}, we obtain
\begin{align}
\|u_h^n\|^2-\|u_h^{n-1}\|^2\leq 2\tau\gamma\|\hat{u}_h^n\|^2,\quad 1\leq n \leq N.\label{11022001}
\end{align}
Therefore, when $\gamma\leq 0$, it arrives that 
\begin{align}
\|u_h^n\|\leq \|u_h^0\|,\quad 1\leq n \leq N.\label{11022101}
\end{align}
When $\gamma>0$, it follows from \eqref{11022001} that 
\begin{align*}
\|u_h^n\|^2-\|u_h^{n-1}\|^2\leq 2\tau\gamma\|\hat{u}_h^n\|^2\leq \frac{\tau\gamma}{2}\left(\|u_h^n\|+\|u_h^{n-1}\|\right)^2.
\end{align*}
A simple transformation gives 
\begin{align*}
\left(1-\frac{\tau\gamma}{2}\right)\|u_h^n\|\leq \left(1+\frac{\tau\gamma}{2}\right)\|u_h^{n-1}\|,\quad 1\leq n \leq N.
\end{align*}
When $\tau\gamma\leq1$, we have 
\begin{align*}
\|u_h^n\|\leq \frac{1+\frac{\tau\gamma}{2}}{1-\frac{\tau\gamma}{2}}\|u_h^{n-1}\|= \left(1+\frac{\tau\gamma}{1-\frac{\tau\gamma}{2}}\right)\|u_h^{n-1}\|\leq\left(1+2\tau\gamma\right)\|u_h^{n-1}\|, \quad 1\leq n \leq N.
\end{align*}
By virtue of the discrete Gronwall inequality given in Lemma \ref{lemma3}, it holds that 
\begin{align}
\|u_h^n\|\leq e^{2\gamma n\tau}\|u_{h}^0\|\leq e^{2\gamma T}\|u_{h}^0\|,\quad 1\leq n \leq N.\label{11022059}
\end{align}
It can be concluded that the theorem holds true by combining inequalities \eqref{11022101} and \eqref{11022059}.
\end{proof}
\begin{remark}
By utilizing the definition of the initial approximate solution $u_{h}^0$ and the approximation property \eqref{2510281415} of the Ritz projection, we can easily obtain that 
\begin{align}
\|u_h^0\|= \|R_hu^0\|\leq \|R_hu^0-u^0\|+\|u^0\|\leq C.\label{11022145}
\end{align}
Using the conclusions of Theorem \ref{theorem2} and \eqref{11022145}, it is known that there exists a positive constant $C$ such that the following equation holds
$$
\|u_h^n\|\leq C,\quad 0\leq n \leq N.
$$
i.e., the numerical solution $\|u_h^n\|$ is bounded in the $L^2$-norm. Furthermore, with the help of the estimate \eqref{2510281415} of the Ritz projection $R_h$, we know that 
\begin{align}
\|\eta_h^n\|=\|R_hu^{n}-u_h^{n}\|\leq \|R_hu^n\|+\|u^n_h\|\leq C, \quad 0\leq n \leq N.\label{11031352}
\end{align}

\end{remark}

To prove the boundedness of the numerical solution under the energy norm, we present the following crucial lemma.
\begin{lemma}\label{lemma10}
Let $u_h^n$ be the solution of the equation \eqref{202510302231}-\eqref{10302242} and $u(t_{n})$ is the solution of the discrete problem \eqref{07271656}. Denote 
\begin{align}
\xi^n=u^{n}-R_hu^{n},\quad \eta_h^n=R_hu^{n}-u_h^{n},\quad 0\leq n \leq N.\notag 
\end{align}
Then there exist a positive constant $\tau_1$ such that when $\tau\leq\tau_1$, we have 
\begin{align}
\|\eta_h^n\|+\tau\|\eta_h^n\|_{\text{DG}}\leq C\left(\tau^2
+h^{k+1}\right), \quad 0 \leq n \leq N.\label{2508292249}
\end{align}
\end{lemma}
\begin{proof}
By virtue of the strong consistency of the interior penalty discontinuous Galerkin scheme \eqref{07271656}, it can be concluded that the exact solution satisfies the following equation at time $t_{n-\frac{1}{2}}$:
\begin{align}
	&(D_\tau u^{n},v_{h})+(\nu + \mathrm{i}\alpha)a_h(\hat{u}^n, v_h)+(\kappa + \mathrm{i}\beta)(\vert \hat{u}^n \vert^2 \hat{u}^n,v_h)-\gamma (\hat{u}^n,v_h)=(R^n,v_h),\quad \forall ~v_h\in V_h^k,\label{11032210}
\end{align}
where 
\begin{align*}
R^n &= D_\tau u^{n}-u_t^{n-\frac{1}{2}} + (\nu + \mathrm{i}\alpha)\left(\Delta u^{n-\frac{1}{2}}-\Delta\hat{u}^{n}\right)\notag\\
&\quad+(\kappa + \mathrm{i}\beta)\left(\vert \hat{u}^n \vert^2 \hat{u}^n-\vert u^{n-\frac{1}{2}} \vert^2 u^{n-\frac{1}{2}}\vert\right)-\gamma \left(\hat{u}^n
-u^{n-\frac{1}{2}}\right),
\end{align*}
and 
\begin{align*}
u_t^{n-\frac{1}{2}}=u_t(t_{n-\frac{1}{2}}),\quad u^{n-\frac{1}{2}} = u(t_{n-\frac{1}{2}}).
\end{align*}
Subtracting \eqref{11032210} from \eqref{07271656} and utilizing the fact $u^{n}=\xi^n+\eta_h^n$, we have the system of error equations
\begin{align}
	&\quad(D_\tau \eta_h^{n},v_{h})+(\nu + \mathrm{i}\alpha)a_h(\hat{\eta}_h^n, v_h)+(\kappa + \mathrm{i}\beta)\left(\vert \hat{u}^n \vert^2 \hat{u}^n-\vert \hat{u}_h^n \vert^2 \hat{u}_h^n,v_h\right)-\gamma (\hat{\eta}_h^n,v_h)\notag\\
	&=(R^n,v_h)-(D_\tau \xi^{n},v_{h})+\gamma (\hat{\xi}^n,v_h),\quad \forall ~v_h\in V_h^k.\label{11032239}
\end{align}
Choosing $v_h=\hat{\eta}_h^n$ in \eqref{11032239} and taking the real part of the resulting equation, it arrives at 
\begin{align}
\frac{1}{2\tau}\left(\|u_h^n\|^2-\|u_h^{n-1}\|^2\right)+\nu\|\hat{\eta}_h^n\|_{a_h}^2&=-\text{Re}\left\{(\kappa + \mathrm{i}\beta)\left(\vert \hat{u}^n \vert^2 \hat{u}^n-\vert \hat{u}_h^n \vert^2 \hat{u}_h^n,\hat{\eta}_h^n\right)\right\}+\gamma\|\hat{\eta}_h^n\|^2\notag\\
&\quad+\text{Re}(R^n,\hat{\eta}_h^n)-\text{Re}(D_\tau\xi^{n},\hat{\eta}_h^n)+\gamma\text{Re}(\hat{\xi}^n,\hat{\eta}_h^n).\label{11032305}
\end{align}
Next, we estimate each term on the right-hand side of equation \eqref{11032305}.  Noticing that 
\begin{align}
\vert\hat{u}^n \vert^2 \hat{u}^n-\vert \hat{u}_h^n \vert^2 \hat{u}_h^n&=\vert\hat{u}^n \vert^2 \left(\hat{\xi}^n+R_h\hat{u}^n\right)-\vert \hat{u}_h^n \vert^2 \left(R_h\hat{u}^n-\hat{\eta}_h^n\right)\notag\\
&= \vert\hat{u}^n \vert^2\hat{\xi}^n+R_h\hat{u}^n\left(\vert\hat{u}^n \vert^2-\vert\hat{u}_h^n \vert^2\right)+\vert\hat{u}_h^n \vert^2\hat{\eta}_h^n\notag\\
&= \vert\hat{u}^n \vert^2\hat{\xi}^n+R_h\hat{u}^n\left(\hat{u}^n (\hat{u}^n -\hat{u}_h^n )^*+(\hat{u}^n -\hat{u}_h^n )(\hat{u}_h^n)^*\right)+\vert\hat{u}_h^n \vert^2\hat{\eta}_h^n\notag\\
&= \vert\hat{u}^n \vert^2\hat{\xi}^n+R_h\hat{u}^n\left(\hat{u}^n (\hat{\xi}^n + \hat{\eta}_h^n )^*+(\hat{\xi}^n + \hat{\eta}_h^n )(R_h\hat{u}^n-\hat{\eta}_h^n)^*\right)+\vert \hat{u}_h^n \vert^2\hat{\eta}_h^n\notag\\
&=\vert\hat{u}^n \vert^2\hat{\xi}^n+R_h\hat{u}^n\left(\hat{u}^n( \hat{\xi}^n)^* + \hat{u}^n(\hat{\eta}_h^n)^*+\hat{\xi}^n(R_h\hat{u}^n)^*
-\hat{\xi}^n(\hat{\eta}_h^n)^*+\hat{\eta}_h^n(R_h\hat{u}^n)^*\right)\notag\\
&\quad-(R_h\hat{u}^n)\hat{\eta}_h^n(\hat{\eta}_h^n)^*+\vert\hat{u}_h^n \vert^2\hat{\eta}_h^n.\notag 
\end{align}
Then we can readily obtain the estimation for the first term on the right-hand side of equation \eqref{11032305}. In fact, 
\begin{align}
&\quad-\text{Re}\left\{(\kappa + \mathrm{i}\beta)\left(\vert \hat{u}^n \vert^2 \hat{u}^n-\vert \hat{u}_h^n \vert^2 \hat{u}_h^n,\hat{\eta}_h^n\right)\right\}\notag\\
&\leq C\|\hat{u}^n\|_{\infty}\|\hat{\xi}^n\|\|\hat{\eta}_h^n\|+\|R_h\hat{u}^n\|_{\infty}\big(\|\hat{u}^n\|_{\infty}\|\hat{\xi}^n\|\|\hat{\eta}_h^n\|+\|\hat{u}^n\|_{\infty}\|\hat{\eta}_h^n\|^2\notag\\
&\quad+\|R_h\hat{u}^n\|_{\infty}\|\hat{\xi}^n\|\|\hat{\eta}_h^n\|+\|\hat{\xi}^n\|_\infty\|\hat{\eta}_h^n\|^2+\|R_h\hat{u}^n\|_{\infty}\|\hat{\eta}_h^n\|^2\big)+\|R_h\hat{u}^n\|_{\infty}\|\hat{\eta}_h^n\|_4^2\|\hat{\eta}_h^n\|\notag\\
&\leq C(\|\hat{\xi}^n\|^2+\|\hat{\eta}_h^n\|^2+\|\hat{\eta}_h^n\|_4^2)\notag\\
&\leq C\left(\|\hat{\xi}^n\|^2+\|\hat{\eta}_h^n\|^2+\|\hat{\eta}_h^n\|_{\text{DG}}\|\hat{\eta}_h^n\|\right),\label{11041542}
\end{align}
where the Cauchy-Schwarz inequality is applied in the first inequality above, the boundedness \eqref{11031352} of $R_h$ in $L^2$-norm is used in the second inequality, and the discrete Ladyzhenskaya's inequality \eqref{202509042305} is employed in the third inequality. 
Next, the Taylor expansion formula with an integral remainder term can be used to obtain
\begin{align}
\text{Re}(R^n,\hat{\eta}_h^n)&\leq C\Big(\|D_\tau u^{n}-u^{n-\frac{1}{2}}\|+\|\Delta u^{n-\frac{1}{2}}-\Delta\hat{u}^{n}\|\notag\\
&\quad+\left\|\vert \hat{u}^n \vert^2 \hat{u}^n-\vert u^{n-\frac{1}{2}} \vert^2 u^{n-\frac{1}{2}}\right\|+\|\hat{u}^n
-u^{n-\frac{1}{2}}\|\Big)\|\hat{\eta}_h^n\|\notag\\
&\leq C(\tau^4+\|\hat{\eta}_h^n\|^2),\label{11041543}
\end{align}
where we have used the following estimates: 
\begin{align}
\|D_\tau u^{n}-u^{n-\frac{1}{2}}\|&= \frac{1}{2\tau}\left\|\int_{t_{n-1}}^{t_{n-\frac{1}{2}}}(s-t_{n-1})^2u_{ttt}(s)\text{d}s+\int_{t_{n-\frac{1}{2}}}^{t_{n}}(s-t_{n})^2u_{ttt}(s)\text{d}s\right\|\notag \leq C\tau^2\notag,
\end{align}
\begin{align}
&\|\Delta u^{n-\frac{1}{2}}-\Delta\hat{u}^{n}\|\leq C\tau \int_{t_{n-1}}^{t_n}\|\Delta u_{tt}\|\text{d}s\leq C\tau^2,\notag
\end{align}
\begin{align}
\left\|\vert \hat{u}^n \vert^2 \hat{u}^n-\vert u^{n-\frac{1}{2}} \vert^2 u^{n-\frac{1}{2}}\right\|&=\left\|\vert \hat{u}^n \vert^2(\hat{u}^n- u^{n-\frac{1}{2}})+\left(\vert \hat{u}^n \vert^2 -\vert u^{n-\frac{1}{2}} \vert^2\right) u^{n-\frac{1}{2}}\right\|\notag\\
&= \left\|\vert \hat{u}^n \vert^2(\hat{u}^n- u^{n-\frac{1}{2}})+\text{Re}\left( (\hat{u}^n-u^{n-\frac{1}{2}})( \hat{u}^n+u^{n-\frac{1}{2}})^*\right) u^{n-\frac{1}{2}}\right\|\notag\\
&\leq C\tau^2,\notag 
\end{align}
and 
\begin{align}
&\| u^{n-\frac{1}{2}}-\hat{u}^{n}\|\leq C\tau \int_{t_{n-1}}^{t_n}\| u_{tt}\|\text{d}s\leq C\tau^2.\notag
\end{align}
In the last, applying the Cauchy-Schwartz inequality and the approximation property \eqref{2510281415}, it follows that 
\begin{align}
-\text{Re}(D_\tau\xi^{n},\hat{\eta}_h^n)+\gamma\text{Re}(\hat{\xi}^n,\hat{\eta}_h^n)&\leq C (\|D_\tau\xi^{n}\|+\|\hat{\xi}^n\|)\|\hat{\eta}_h^n\|\notag\\
&=C\left(\frac{1}{\tau}\left\|\int_{t_{n-1}}^{t_n}(R_h(u_t)-u_t)\text{d}s\right\|+\|\hat{\xi}^n\|\right)\|\hat{\eta}_h^n\|\notag\\ 
&\leq C\left(h^{2k+2}+\|\hat{\eta}_h^n\|^2\right).\label{11041540}
\end{align}
Substituting \eqref{11041542}–\eqref{11041540} into \eqref{11032305} and multiplying both sides of the resulting equation by \(2\tau\) yields
\begin{align*}
\|\eta_h^n\|^2-\|\eta_h^{n-1}\|^2+2\nu C_{1}\tau\|\hat{\eta}_h^n\|_{\text{DG}}^2&\leq C\left(\tau^5+\tau h^{2k+2}+\tau\|\hat{\eta}_h^n\|^2+\tau\|\hat{\eta}_h^n\|_{\text{DG}}\|\hat{\eta}_h^n\|\right)\notag\\
&\leq C\left(\tau^5+\tau h^{2k+2}+\tau\|\hat{\eta}_h^n\|^2\right)+\nu C_1\tau
\|\hat{\eta}_h^n\|_{\text{DG}}^2.
\end{align*}
Combining like terms, we have 
\begin{align}
\|\eta_h^n\|^2-\|\eta_h^{n-1}\|^2+\nu C_{1}\tau\|\hat{\eta}_h^n\|_{\text{DG}}^2
\leq C\left(\tau^5+\tau h^{2k+2}+\tau\|\hat{\eta}_h^n\|^2\right), \quad 1\leq n \leq N.\label{11041612}
\end{align}
Replacing $n$ by $j$ in \eqref{11041612} and summing over $j$ from $1$ to $n$, we get 
\begin{align*}
\|\eta_h^n\|^2+\nu C_{1}\tau\sum\limits_{j=1}^{n}\|\hat{\eta}_h^j\|_{\text{DG}}^2\leq C(\tau^4+h^{2k+2})+ C\tau\sum\limits_{j=1}^{n}\|\eta_h^j\|^2, \quad 1\leq n \leq N.
\end{align*}
According to the Gronwall inequality presented in Lemma \ref{lemma3}, it holds that 
\begin{align}
\|\eta_h^n\|^2+\nu C_{1}\tau\sum\limits_{j=1}^{n}\|\hat{\eta}_h^j\|_{\text{DG}}^2\leq C(\tau^4+h^{2k+2}). \label{11041641}
\end{align}
Employing the transfer inequality \eqref{2507302329}, \eqref{11041641} and the Cauchy-Schwartz inequality give that 
\begin{align}
\tau\|\eta_h^n\|_{\text{DG}}\leq 2\tau \sum\limits_{j=1}^{n}\|\hat{\eta}_h^j\| \leq C{\sqrt{n\tau}}\sqrt{\sum\limits_{j=1}^{n}\nu C_{1}\tau\|\hat{\eta}_h^j\|_{\text{DG}}^2}\leq C(\tau^2+h^{k+1}).\label{123456}
\end{align}
In combination with \eqref{11041641}, it follows that the conclusion in \eqref{2508292249} is true. 
\end{proof}
With the above preparations, the boundedness of the numerical solution under the energy norm follows immediately.
\begin{theorem}
The solution $u_{h}^{n}$ of the fully discrete algorithm \eqref{07271656} is bounded in $\text{DG}$-norm, i.e., there exists a positive constant $C$ such that  
\begin{align*}
\|u_h^n\|_{\text{DG}}\leq C,\quad 0\leq n \leq N. 
\end{align*}
\end{theorem}
\begin{proof}
The core idea of the argument is to conduct a classified discussion on the spatial parameter $h$ and the temporal parameter $\tau$. when $\tau\leq h$,  it follows from the global inverse inequality \eqref{10312320} and \eqref{2508292249} that 
\begin{align}
\|\eta_h^n\|_{\text{DG}}\leq Ch^{-1}\|\eta_h^n\|\leq Ch\leq C.\label{11041715}
\end{align}
When $h<\tau$, we can easily obtain from  \eqref{2508292249} that 
\begin{align}
\|\eta_h^n\|_{\text{DG}}\leq C\tau^{-1} (\tau^2+h^{k+1})\leq C\tau\leq C.\label{11041719}
\end{align}
Through the combination of \eqref{11041715} and \eqref{11041719} and the fact that $\|R_hu^n\|_{\text{DG}}\leq C\|u^n\|_{\text{DG}}$, the truth of the theorem can be established.
\end{proof}
Using the boundedness of the numerical solution under the energy norm, the uniqueness of the numerical solution follows immediately.
\begin{theorem}
The solution of the discrete Galerkin scheme \eqref{07271656} is unique for sufficiently small $\tau$.
\end{theorem}
\begin{proof}
Adopt the background in Lemma \ref{theorem1}. To prove the uniqueness of the numerical solution, suppose that $w_1$ and $w_2$ satisfy
\begin{align}
&\frac{2}{\tau}\left(w_1-u_{h}^{n-1},v_{h}\right)+(\nu + \mathrm{i}\alpha)a_h(w_1, v_h)+(\kappa + \mathrm{i}\beta)(\vert w_1 \vert^2 w_1,v_h)-\gamma (w_1,v_h)=0, \quad \forall ~v_h\in V_h^k,\label{11042111}\\
&\frac{2}{\tau}\left(w_2-u_{h}^{n-1},v_{h}\right)+(\nu + \mathrm{i}\alpha)a_h(w_2, v_h)+(\kappa + \mathrm{i}\beta)(\vert w_2 \vert^2 w_2,v_h)-\gamma (w_2,v_h)=0, \quad \forall ~v_h\in V_h^k,\label{11042112}
\end{align}
respectively. Subtracting \eqref{11042112} from  \eqref{11042111}, taking $v_h=w_1-w_2$ and considering the real part on both sides of the corresponding equation give 
\begin{align}
&\quad\frac{2}{\tau} \|w_1-w_2\|^2+\nu\|w_1-w_2\|_{a_h}^2\notag\\
&=-Re\left((\kappa + \mathrm{i}\beta)(\vert w_1 \vert^2 w_1-\vert w_2 \vert^2 w_2,w_1-w_2) \right)+\gamma\|w_1-w_2\|^2\notag\\
&\leq C \left((\|w_1\|_{8}^2+\|w_1\|_8\|w_2\|_8+\|w_2\|_8^2)\|w_1-w_2\|_{4}\|w_1-w_2\|+\|w_1-w_2\|^2\right)\notag\\
&\leq C \left((\|w_1\|_{\text{DG}}^2+\|w_1\|_{\text{DG}}\|w_2\|_{\text{DG}}+\|w_2\|_{\text{DG}}^2)\|w_1-w_2\|_{\text{DG}}^\frac{1}{2}\|w_1-w_2\|^{\frac{3}{2}}+\|w_1-w_2\|^2\right)\notag\\
&\leq C \|w_1-w_2\|^2+\nu\|w_1-w_2\|_{a_h}^2,\label{11042322}
\end{align}
where we have used the elementary inequality \eqref{202511042246} and Hölder inequality in the first inequality, followed by the imbedding inequality \eqref{11042314} and discrete Ladyzhenskaya's inequality \eqref{202509042305} in the second equality, and the Young inequality in the third inequlity. 
From \eqref{11042322}, we can immediately obtain
\begin{align*}
\|w_1-w_2\|\leq C\tau \|w_1-w_2\|.
\end{align*} 
Therefore, for a suitably small $\tau$, we have $w_1=w_2$, i.e., the scheme \eqref{07271656} is uniquely solvable.
\end{proof}
In tha last, we establish the unconditional optimal error estimates of the numerical solution in both the \(L^2\)-norm and the energy norm.
\begin{theorem}
Let $u^n$ and $u_h^n$ be the solutions of \eqref{202510302231}-\eqref{10302242} and \eqref{07271656}, respectively, then there exist two positive constants $\tau_0$ and $h_0$ such that when $\tau\leq\tau_1$, $h\leq h_1$, it holds that 
\begin{align}
\|u^{n}-u_{h}^{n}\|\leq C\left(\tau^2+h^{k+1}\right),\quad \|u^{n}-u_{h}^{n}\|_{\text{DG}}\leq C\left(\tau^2+h^{k}\right).
\end{align}
\end{theorem}
\begin{proof}
We first present the optimal order error estimate in the $L^2$-norm. For this purpose, utilizing the triangle inequality and Lemma \ref{lemma10}, we have 
\begin{align}
\|u^{n}-u_{h}^{n}\|\leq \|\xi^n\|+\|\eta^n_{h}\|\leq C\left(\tau^2+h^{k+1}\right).\label{11051049}
\end{align}
Next, taking $v_h=D_\tau \eta_h^{n}$ in \eqref{11032239} and the real part together with the similar procedure as in Lemma \ref{lemma10} lead to 
\begin{align}
&\quad\|D_\tau \eta_h^{n}\|^2+\frac{\nu}{2\tau}\left(\|\eta_h^n\|_{a_h}^2-\|\eta_h^{n-1}\|_{a_h}^2\right)\notag\\
&=-\text{Re}\left\{(\kappa + \mathrm{i}\beta)\left(\vert \hat{u}^n \vert^2 \hat{u}^n-\vert \hat{u}_h^n \vert^2 \hat{u}_h^n,D_\tau \eta^{n}_h\right)\right\}-\gamma\left(\hat{\eta}_h^n,D_\tau \eta_h^{n}\right)\notag\\
&\quad+\text{Re}(R^n,D_\tau \eta_h^{n})-\text{Re}(D_\tau \xi^{n},D_\tau \eta_h^{n})+\gamma \text{Re}(\hat{\xi}^n,D_\tau \eta_h^{n})\notag\\
&\leq C\left(\tau^4+h^{2k+2}\right)+\|D_\tau \eta_h^{n}\|^2+C\|\eta_h^n\|_{\text{DG}}^2.\label{11052141}
\end{align}
It should be emphasized here that we used the $L^2$-norm error estimation \eqref{11051049} in the above argument. 
Multiplying both sides of inequality \eqref{11052141} by $2\tau$ and summing up, we get 
\begin{align*}
\|\eta_h^n\|_{\text{DG}}\leq C\tau\sum\limits_{m=1}^{n} \|\eta_h^m\|_{\text{DG}}^2+C\left(\tau^2+h^{2k+2}\right).
\end{align*}
Applying the discrete Gronwall inequality, we have 
\begin{align*}
\|\eta_h^n\|_{\text{DG}}\leq C\left(\tau^2+h^{k+1}\right).
\end{align*}
It follows from the triangle inequality that 
\begin{align*}
\|u^n-u_h^n\|_{\text{DG}}\leq \|\xi^n\|_{\text{DG}}+\|\eta_h^n\|_{\text{DG}}\leq C\left(\tau^2+h^{k}\right).
\end{align*}
The proof is completed.
\end{proof}
\section{Numerical examples}\label{section4}
In this section, we present two numerical examples to validate the theoretical findings of this work. All numerical results are implemented using the open-source finite element software FreeFEM \cite{2021Hecht}.

\begin{example}\label{example01}
Consider the following Ginzburg-Landau equation
\begin{align}
u_t-(1 + \mathrm{i})\Delta u + (1 + \mathrm{i})\vert u \vert^2 u -  u = f(x,y,t), \quad (x,y) \in (0,1) \times (0,1),~t\in (0,T], \notag 
\end{align}
with the exact solution 
\begin{align}
u(x,y,t) = \sin(\pi x)\sin(\pi y)\text{e}^{\text{i}t^2},\notag
\end{align}
and the source function $f(x,y,t)$ is determined by the above exact solution.

The numerical results of this example are presented in Tables \ref{table1}–\ref{table4}. To test the convergence orders in the spatial direction, we select $T = 2 \times 10^{-6}$, $N = 20$ and  the spatial step sizes $h=1/5,1/10,1/15,1/20,1/25$ under different polynomial degrees $k=1,2,3$.
In Tables \ref{table1} to \ref{table3}, the errors and optimal convergence accuracies under different norms are presented. As evident from the tabulated numerical results, it can be observed that the convergence orders of the errors in the $L^2$-norm and \text{DG}-norm both achieve the theoretically expected accuracies \(\mathcal{O}(h^{k+1})\) and \(\mathcal{O}(h^{k})\) as derived earlier. To test the convergence orders in the temporal direction, we set \(T = 1\), \(\tau = h\) and $k=3$, and the relevant computational results are presented in Table \ref{table4}. The results suggest that the accuracy in the temporal direction achieves second-order convergence. The above results in the spatial and temporal directions fully demonstrate the effectiveness and robustness of the algorithm proposed in this paper. 

\begin{table}
	\centering
	\caption{The spatial convergence orders with $k=1$ at the final time $T=2\text{e}-6$.}\label{table1}
	\begin{tabular}{ccccc}
		\toprule 
		$h$ & $\|u^N-u_h^N\|$ & $\text{Order}$ & $\|u^N-u_h^N\|_{\text{DG}}$ & $\text{Order}$\\
		\midrule
		1/5 &   5.0867e-02 &  --      &  6.7062e-01     & --\\
		1/10 &  1.3085e-02 &  1.9588  & 3.3807e-01  & 0.9882\\
		1/15 &  5.7763e-03 &  2.0167  & 2.2379e-01  &1.0175 \\
		1/20 &  3.2100e-03 &  2.0421  & 1.6631e-01  & 1.0318\\
		1/25 &  2.0268e-03 &  2.0608  & 1.3178e-01  & 1.0428\\
		\bottomrule
	\end{tabular}
\end{table}
\begin{table}
	\centering
	\caption{The spatial convergence orders with $k=2$ at the final time $T=2\text{e}-6$.}\label{table2}
	\begin{tabular}{ccccc}
		\toprule 
		$h$ & $\|u^N-u_h^N\|$ & $\text{Order}$ & $\|u^N-u_h^N\|_{\text{DG}}$ & $\text{Order}$\\
		\midrule
		1/5 &   1.8174e-03 &  --      &  8.1516e-02     & --\\
		1/10 &  2.1470e-04 &  3.0815  & 2.0340e-02  & 2.0027\\
		1/15 &  6.0179e-05 &  3.1369  & 8.9062e-03  &2.0368 \\
		1/20 &  2.4057e-05 &  3.1871  & 4.9354e-03 & 2.0520\\
		1/25 &  1.1687e-05 &  3.2354 & 3.1162e-03  & 2.0606\\
		\bottomrule
	\end{tabular}
\end{table}
\begin{table}
	\centering
	\caption{The spatial convergence orders with $k=3$ at the final time $T=2\text{e}-6$.}\label{table3}
	\begin{tabular}{ccccc}
		\toprule 
		$h$ & $\|u^N-u_h^N\|$ & $\text{Order}$ & $\|u^N-u_h^N\|_{\text{DG}}$ & $\text{Order}$\\
		\midrule
		1/5 &   1.6529e-04 &  --      &  3.2644e-03     & --\\
		1/10 &  9.9187e-06 &  4.0587  & 3.9142e-04 & 3.0600\\
		1/15 &  1.9195e-06 &  4.0505  & 1.1472e-04  &3.0269 \\
		1/20 &  5.9857e-07 &  4.0507  & 4.8330e-05  & 3.0049\\
		1/25 &  2.4216e-07 &  4.0555  & 2.4829e-05  & 2.9848\\
		\bottomrule
	\end{tabular}
\end{table}
\begin{table}
	\centering
	\caption{The temporal convergence orders with $k=3$ at the final time $T=1$.}\label{table4}
	\begin{tabular}{ccccc}
		\toprule 
		$\tau$ & $\|u^N-u_h^N\|$ & $\text{Order}$ & $\|u^N-u_h^N\|_{\text{DG}}$ & $\text{Order}$\\
		\midrule
		1/20 &   1.8177e-05 &  --      &  9.8640e-05     & --\\
		1/25 &   1.1607e-05 &  2.0100  & 6.0324e-05 & 2.2037\\
		1/30 &   8.0512e-06 &  2.0063  & 4.0810e-05  &2.1435 \\
		1/35 &   5.9112e-06 &  2.0044  & 2.9499e-05  & 2.1056\\
		1/40 &   4.5238e-06 &  2.0032  & 2.2344e-05  & 2.0804\\
		\bottomrule
	\end{tabular}
\end{table}
\end{example}
\begin{example}\label{example02}
In \eqref{202510302231}-\eqref{10302242}, let $\alpha=\beta=\gamma=\nu=\kappa=1$ and $\Omega=(-1,1)\times(-1,1)$. The analytical solution is taken as 
\begin{align}
u(x,y,t) = (1+x)^4(1-x)^4(1+y)^4(1-y)^4\text{e}^{\text{i}t}.\notag
\end{align}
From the above exact solution, the initial condition and the right-hand side function can be directly derived.

The results of the second numerical example are presented in Tables \ref{table5} to \ref{table8}.
They are in excellent agreement with the theoretical arguments. In the test of convergence accuracy in the temporal direction, we fix \(h = 1/50\) and vary the time step size, and the numerical results achieve the expected second-order accuracy. The above numerical results further verify the correctness of the theoretical findings.  

\begin{table}
	\centering
	\caption{The spatial convergence orders with $k=1$ at the final time $T=2\text{e}-6$.}\label{table5}
	\begin{tabular}{ccccc}
		\toprule 
		$h$ & $\|u^N-u_h^N\|$ & $\text{Order}$ & $\|u^N-u_h^N\|_{\text{DG}}$ & $\text{Order}$\\
		\midrule
		2/10 &   3.7771e-02 &  --      & 4.9915e-01     & --\\
		2/15 &   1.7320e-02 &  1.9229  & 3.3706e-01 & 0.9684\\
		2/20 &   9.8114e-03 &  1.9756  & 2.5324e-01  &0.9939 \\
		2/25 &   6.2780e-03 &  2.0009  & 2.0231e-01  &1.0063\\
		2/30 &   4.3464e-03 &  2.0168  & 1.6815e-01  &1.0143\\
		\bottomrule
	\end{tabular}
\end{table}
\begin{table}
	\centering
	\caption{The spatial convergence orders with $k=2$ at the final time $T=2\text{e}-6$.}\label{table6}
	\begin{tabular}{ccccc}
		\toprule 
		$h$ & $\|u^N-u_h^N\|$ & $\text{Order}$ & $\|u^N-u_h^N\|_{\text{DG}}$ & $\text{Order}$\\
		\midrule
		2/10 &   1.6974e-03 &  --      & 6.8834e-02     & --\\
		2/15 &   4.9250e-04 &  3.0517  & 3.0846e-02 & 1.9797\\
		2/20 &   2.0308e-04 &  3.0794  & 1.7317e-02  &2.0069 \\
		2/25 &   1.0153e-04 &  3.1065  & 1.1027e-02  &2.0228\\
		2/30 &   5.7362e-05 &  3.1319  & 7.6110e-03  &2.0332\\
		\bottomrule
	\end{tabular}
\end{table}

\begin{table}
	\centering
	\caption{The spatial convergence orders with $k=3$ at the final time $T=2\text{e}-6$.}\label{table7}
	\begin{tabular}{ccccc}
		\toprule 
		$h$ & $\|u^N-u_h^N\|$ & $\text{Order}$ & $\|u^N-u_h^N\|_{\text{DG}}$ & $\text{Order}$\\
		\midrule
		2/10 &   1.6237e-04 &  --      & 3.5477e-03     & --\\
		2/15 &   3.2170e-05 &  3.9925  & 1.0471e-03 & 3.0095\\
		2/20 &   1.0154e-05 &  4.0083  & 4.4184e-04  &2.9993 \\
		2/25 &   4.1470e-06 &  4.0132  & 2.2656e-04  &2.9933\\
		2/30 &   1.9939e-06 &  4.0166  & 1.3139e-04  &2.9883\\
		\bottomrule
	\end{tabular}
\end{table}

\begin{table}
	\centering
	\caption{The temporal convergence orders with $k=3$ at the final time $T=1$.}\label{table8}
	\begin{tabular}{ccccc}
		\toprule 
		$\tau$ & $\|u^N-u_h^N\|$ & $\text{Order}$ & $\|u^N-u_h^N\|_{\text{DG}}$ & $\text{Order}$\\
		\midrule
		1/10 &   2.8739e-04 &  --      &  7.1514e-04     & --\\
		1/15 &   1.2805e-04 & 1.9938   & 3.2113e-04 & 1.9745\\
		1/20 &   7.2071e-05 & 1.9979  & 1.8064e-04  & 2.0000\\
		1/25 &   4.6139e-05 & 1.9987   & 1.1568e-04  & 1.9971\\
		1/30 &   3.2046e-05 & 1.9991   & 8.0405e-05  & 1.9952\\
		\bottomrule
	\end{tabular}
\end{table}
\end{example}

\section{Conclusions}\label{section5}
Based on the Crank-Nicolson scheme, this paper proposes a fully implicit discontinuous Galerkin fully discrete scheme. The primary innovation of this work resides in the development of a novel analytical technique, through which the unique solvability and optimal error estimates of the numerical solution are rigorously derived. This work’s proof framework can be effectively utilized as a reference for analyzing fully implicit numerical schemes applied to a broader class of nonlinear partial differential equations. For future research, several promising directions are worthy of exploration. First, extending the proposed analytical reasoning technique to three-dimensional settings merits further investigation. Second, this study only employs the second-order Crank-Nicolson scheme for analytical purposes. In fact, it remains an open question whether the developed analytical method can be generalized to high-order temporal discretization schemes, such as the spectral collocation method and the discontinuous Galerkin time-stepping method.

\section*{Declaration of competing interest}
The authors declare that they have no known competing financial interests or personal relationships that could have appeared to influence the work reported in this paper.

\section*{Data availability}
Data will be made available on request.

\section*{Acknowledgments}
This work is supported by the Doctoral Starting Foundation of Pingdingshan University (No. PXY-BSQD2023022) and the Natural Science Foundation of Henan Province (Nos. 242300420655, 252300420343).


\begin{thebibliography}{00}

\bibitem{Guo2020}
Guo B, Jiang M, Li Y. Ginzburg-Landau equations. Science Press, 2020. 

\bibitem{Chen2023}
Chen F, Li M, Zhao Y, et al. Convergence and superconvergence analysis of finite element methods for nonlinear Ginzburg-Landau equation with Caputo derivative. Computational and Applied Mathematics, 2023, 42(6): 271.

\bibitem{Yao2026}
Yao C, Li L, Zhao Y. An energy stable implicit-explicit finite element method/BDF-$k$ scheme for the Ginzburg-Landau equation. Communications in Nonlinear Science and Numerical Simulation, 2025: 109284.

\bibitem{Zhang2025}
Zhang J, Yang X, Wang S. A three-layer FDM for the Neumann initial-boundary value problem of 2D Kuramoto-Tsuzuki complex equation with strong nonlinear effects. Communications in Nonlinear Science and Numerical Simulation, 2025: 109255.

\bibitem{Gao2023}
Gao H, Xie W. A finite element method for the dynamical Ginzburg-Landau equations under Coulomb gauge. Journal of Scientific Computing, 2023, 97(1): 19.

\bibitem{Ma2023}
Ma L, Qiao Z. An energy stable and maximum bound principle preserving scheme for the dynamic Ginzburg-Landau equations under the temporal gauge. SIAM Journal on Numerical Analysis, 2023, 61(6): 2695-2717.

\bibitem{ZhangYang2025}
Zhang J, Yang X. A three‐layer finite difference scheme for the nonlinear Kuramoto‐Tsuzuki complex equation with variable coefficients. Mathematical Methods in the Applied Sciences, 2025.

\bibitem{XuChang2011}
Xu Q, Chang Q. Difference methods for computing the Ginzburg‐Landau equation in two dimensions. Numerical Methods for Partial Differential Equations, 2011, 27(3): 507-528.

\bibitem{HuChenChang2015}
Hu X, Chen S, Chang Q. Fourth‐order compact difference schemes for 1D nonlinear Kuramoto–Tsuzuki equation. Numerical Methods for Partial Differential Equations, 2015, 31(6): 2080-2109.

\bibitem{Shi2020}
Shi D, Liu Q. Superconvergence analysis of a two grid finite element method for Ginzburg-Landau equation. Applied Mathematics and Computation, 2020, 365: 124691.

\bibitem{Yang2025}
Yang H, Jia X. Unconditionally convergence and superconvergence error analyses of the backward Euler-Galerkin finite element method for the Kuramoto-Tsuzuki equation. International Journal of Computer Mathematics, 2025: 1-21.

\bibitem{LiCaoZhang2019}
Li D, Cao W, Zhang C, et al. Optimal error estimates of a linearized Crank-Nicolson Galerkin FEM for the Kuramoto-Tsuzuki equations. Communications in Computational Physics, 2019, 26(3): 838-854.

\bibitem{Li2025}
Li X, Cui X, Zhang S. Analysis of a Crank-Nicolson fast element-free Galerkin method for the nonlinear complex Ginzburg-Landau equation. Journal of Computational and Applied Mathematics, 2025, 457: 116323.

\bibitem{GuanCao2025}
Guan Z, Cao X. Weighted implicit-explicit discontinuous Galerkin methods for two-dimensional Ginzburg-Landau equations on general meshes. arXiv preprint arXiv:2510.10283, 2025.


\bibitem{Li2012}
Li B, Sun W. Error analysis of linearized semi-implicit Galerkin finite element methods for nonlinear parabolic equations. arXiv preprint arXiv:1208.4698, 2012.


\bibitem{Sun2017}
Sun W, Wang J. Optimal error analysis of Crank-Nicolson schemes for a coupled nonlinear Schrödinger system in 3D. Journal of Computational and Applied Mathematics, 2017, 317: 685-699.

\bibitem{FengLiMa2021}
Feng X, Li B, Ma S. High-order mass- and energy-conserving SAV-Gauss collocation finite element methods for the nonlinear Schrödinger equation. SIAM Journal on Numerical Analysis, 2021, 59(3): 1566-1591.

\bibitem{GuoChenZhou2025}
Guo J, Chen Y, Zhou J, et al. Optimal error estimates of conservative virtual element method for the coupled nonlinear Schrödinger-Helmholtz equation. Communications in Nonlinear Science and Numerical Simulation, 2025, 145: 108680.

\bibitem{LiLiSun2023}
Li D, Li X, Sun H. Optimal error estimates of SAV Crank-Nicolson finite element method for the coupled nonlinear Schrödinger equation. Journal of Scientific Computing, 2023, 97(3): 71.

\bibitem{WangWang2023}
Wang T, Wang T. Optimal pointwise error estimates of two conservative finite difference schemes for the coupled Gross-Pitaevskii equations with angular momentum rotation terms. Journal of Computational and Applied Mathematics, 2023, 425: 115056.

\bibitem{PietroErn2010}
Di Pietro D, Ern A. Discrete functional analysis tools for discontinuous Galerkin methods with application to the incompressible Navier-Stokes equations. Mathematics of Computation, 2010, 79(271): 1303-1330.

\bibitem{GazcaOrozcoKaltenbach2025}
Gazca-Orozco P A, Kaltenbach A. On the stability and convergence of discontinuous Galerkin schemes for incompressible flows. IMA Journal of Numerical Analysis, 2025, 45(1): 243-282.

\bibitem{Chave2016}
Chave F, Di Pietro D A, Marche F, et al. A hybrid high-order method for the Cahn-Hilliard problem in mixed form. SIAM Journal on Numerical Analysis, 2016, 54(3): 1873-1898.

\bibitem{DiPietroErn2012}
Di Pietro D A, Ern A. Mathematical aspects of discontinuous Galerkin methods. Springer Science $\&$ Business Media, 2011.


\bibitem{Riviere2008}
Rivière B. Discontinuous Galerkin methods for solving elliptic and parabolic equations: theory and implementation. Society for Industrial and Applied Mathematics, 2008.
\bibitem{Akrivis1991}

Akrivis G D, Dougalis V A, Karakashian O A. On fully discrete Galerkin methods of second-order temporal accuracy for the nonlinear Schrödinger equation. Numerische Mathematik, 1991, 59(1): 31-53.

\bibitem{Akrivis1993}
Akrivis G D. Finite difference discretization of the cubic Schrödinger equation. IMA Journal of Numerical Analysis, 1993, 13(1): 115-124.

\bibitem{HeywoodRannacher1990}
Heywood J G, Rannacher R. Finite element approximation of the nonstationary Navier-Stokes problem. Part IV: error analysis for second-order time discretization. SIAM Journal on Numerical Analysis, 1990, 27(2): 353-384.


\bibitem{Sun2023}
Sun Z Z, Zhang Q, Gao G. Finite difference methods for nonlinear evolution equations. Walter de Gruyter GmbH $\&$ Co KG, 2023.

\bibitem{Wang2014}
Wang J. A new error analysis of Crank-Nicolson Galerkin FEMs for a generalized nonlinear Schrödinger equation. Journal of Scientific Computing, 2014, 60(2): 390-407.

\bibitem{2021Hecht}
Hecht F. New development in FreeFem++. Journal of numerical mathematics, 2012, 20(3-4): 1-14.
\end{thebibliography}
\end{document}